\documentclass{article}
\usepackage[all]{xy}
\usepackage[dvips]{graphicx}
\usepackage{amsmath,amssymb}
\usepackage{amsthm}
\newtheorem{theorem}{Theorem}[section]
\newtheorem{prop}[theorem]{Proposition}
\newtheorem{lemma}[theorem]{Lemma}

\newtheorem{cor}[theorem]{Corollary}

\theoremstyle{definition}
\newtheorem{defini}[theorem]{Definition}

\newtheorem{remark}[theorem]{Remark}

\newcommand{\RR}{\mathbb R}

\newcommand{\ZZ}{\mathbb Z}

\newcommand{\Sign}{{\rm Sign}}

\usepackage{amsmath,amssymb}
\usepackage{pinlabel} 
\usepackage{amscd}
\usepackage[all]{xy}

\newcommand{\tr}{{\rm Tr}^{\boxtimes 2}}
\newcommand{\Zdam}{Z_{O\hspace{-.175em}-\hspace{-.175em}O}}

\numberwithin{equation}{section}
\begin{document}
\begin{center}

{\huge A note on the $\Theta$-invariant of 3-manifolds}
\vskip15mm

{\large 
ALBERTO S. CATTANEO\footnote{University of Zurich, email: cattaneo@math.uzh.ch} and  TATSURO SHIMIZU\footnote{Research Institute for Mathematical Sciences, Kyoto University, \\email: shimizu@kurims.kyoto-u.ac.jp}}

\end{center}
\vskip5mm
\begin{abstract}
In this note, we revisit the $\Theta$-invariant as defined by 
R.~Bott and the first author in \cite{BC2}.
The $\Theta$-invariant is an invariant of rational homology 3-spheres with acyclic orthogonal local systems, which is a generalization of the 2-loop term of the Chern-Simons perturbation theory.
The $\Theta$-invariant can be defined when a cohomology group is vanishing.
In this note, we give a slightly modified version of the $\Theta$-invariant that we can define even if the cohomology group is not vanishing.
\end{abstract}
\section{Introduction}\label{intro} 
In 1998, R.~Bott and the first author defined topological invariants of rational homology 
spheres with acyclic orthogonal local systems in \cite{BC1}, \cite{BC2}.
These invariant were inspired by the Chern-Simons perturbation theory developed by M.~Kontsevich in \cite{Kon}, S.~Axelrod and M.~I.~Singer in \cite{AS}.
The Chern-Simons perturbation theory gives invariants of 
3-manifolds with flat connections of the trivial $G$-bundle over the 3-manifold,
where $G$ is a semi-simple Lie group.
The composition of adjoint representation of $G$ and the holonomy representation 
of the flat connection gives an orthogonal local system.

In \cite{BC2}, Bott and the first author constructed a real valued invariant, called $\Theta$-invarant (In this note, we denote by $Z_{\Theta}$ the corresponding term), which is a generalization of a 2-loop term of Chern-Simons perturbation theory.
The vanishing of a cohomology group 
(denoted by $H^*_-(\Delta;\pi_1^{-1}E\otimes \pi_2^{-1}E)$ in \cite{BC2}, $H^*_-(\Delta;E_{\rho}\boxtimes E_{\rho})$ in this note) plays an important role
in the construction of the $\Theta$-invariant $Z_{\Theta}$.
There are few gaps in the proof of this vanishing (Lemma 1.2 of \cite{BC2}).
In this note, we show that a linear combination of $Z_{\Theta}$ and another term $\Zdam$ is, however, a topological invariant of closed 3-manifolds with orthogonal acyclic local systems, when the local system is given by using a holonomy representation of a flat connection.
The term $\Zdam$ is also related to the 2-loop term of the Chern-Simons perturbation theory.
We note that the second author proved that when $G=SU(2)$, $Z_{\Theta}$ itself is an invariant of closed 3--manifolds with orthogonal local systems in \cite{Shimizu2}.

The organization of this paper is as follows.
In Section~2 we give a modified version of
the Bott-Cattaneo $\Theta$-invariant without proof.
In Section~3 and Section~4 we prove a proposition and a theorem about well-definedness of the invariant stated in Section~2.
Both the invariant defined in Section~2 of this note and
the $\Theta$-invariant depend on the choice of a framing of the 3-manifold. 
In Section 5 we introduce a framing correction.

\subsection*{Orientation convention}
In this note, all manifolds are oriented.
Boundaries are oriented by the outward normal first convention.
Products of oriented manifolds are oriented by the order of the factors.
The interval $[0,1]\subset \RR$ is oriented from $0$ to $1$.
\subsection*{Acknowledgments}
A. S. C. acknowledges partial support of SNF Grant No. 200020 172498/1.
This research was (partly) supported by the NCCR SwissMAP,  funded by the Swiss
National Science Foundation, and by the COST Action MP1405 QSPACE, supported by
COST (European Cooperation in Science and Technology).
T.~S. expresses his
appreciation to Professor Tadayuki Watanabe for his
helpful comments and discussion on the Chern-Simons perturbation theory.
This work was (partly) supported by JSPS KAKENHI Grant Number JP15K13437.
\section{The invariant}
Let $M$ be a closed oriented framed 3-manifold,
namely a trivialization of the tangent bundle of $M$ is fixed.
We take a metric on $M$ compatible with the framing.
Let $\rho: \pi_1\to G$ be a representation of the fundamental group into a semi-simple Lie group $G$.
We denote by ${\rm ad}:G\to {\mathfrak g}$ the adjoint representation of $G$, where $\mathfrak g$ is the Lie algebra of $G$.
Since $G$ is semi-simple, the representation ${\rm ad}\circ \rho$ is orthonormal with respect to the Killing form.
A local system is a covariant functor from the fundamental groupoid of $M$ to the category of finite 
dimensional vector spaces.
Note that a representation of $\pi_1(M)$ gives a local system.  
We denote by $E_{\rho}$ the local system given by ${\rm ad}\circ \rho$. 
We assume that $E_{\rho}$ is acyclic, namely
$$H^*(M;E_{\rho})=0.$$
In this note, we say that such a representation $\rho$ is {\it acyclic}.
\subsection{A compactification of a configuration space}\label{sec21}
Let $\Delta=\{(x,x)\mid x\in M\}\subset M^2$ be the diagonal.
We identify $\Delta$ with $M$ by
$$\Delta\ni (x,x)\to x\in M.$$
We orient $\Delta$ by using this identification.
We denote by $\nu_{\Delta}$ the normal bundle of $\Delta$ in $M^2$.
We identify $\nu_{\Delta}$ with the tangent bundle $TM$ via the isomorphism defined by
$$TM\stackrel{\cong}{\to} \nu_{\Delta}, (x,v)\mapsto ((x,x),(-v,v))$$
where $x\in M$ and $v\in T_xM$.
On the other hand, $M$ is framed.
Then $TM$ is identified with $M\times\RR^3$.
Thus $\nu_{\Delta}$ is identified with $M\times \RR^3$.
  
Let $C_2(M)=B\ell(M^2,\Delta)$ be the compact 6-dimensional manifold
with the boundary obtained by the
real blowing up of $M^2$ along $\Delta$.
We denote by 
$$q:C_2(M)\to M^2$$
the blow-down map.
As manifolds, $$C_2(M)=(M^2\setminus\Delta)\cup S\nu_{\Delta}$$
and 
$q(S\nu_{\Delta})=\Delta$.
Here $S\nu_{\Delta}$ is the unit sphere bundle of $\nu_{\Delta}$
with respect to the metric on $M$.
$C_2(M)$ is a compactification of the configuration space $M^2\setminus \Delta$ of two distinct points. Obviously, $\partial C_2(M)=S\nu_{\Delta}$.

$S\nu_{\Delta}$ is identified with $\Delta\times S^2$.
We denote by 
$$p:\partial C_2(M)=\Delta\times S^2\to S^2$$
the projection.
We use the same symbol $q$ for 
the restriction map $q|_{\partial C_2(M)}:\partial C_2(M)(=\Delta\times S^2)\to \Delta$ of the blow-down map $q$.   

\subsection{The natural transformations $c$ and ${\rm Tr}$}
The killing form gives an isomorphism ${\mathfrak g}\otimes {\mathfrak g}\cong {\mathfrak g}^*\otimes {\mathfrak g}^*$.
Let $\bf 1\in \mathfrak g\otimes \mathfrak g$ the element corresponding to the killing form in $\mathfrak g^*\otimes \mathfrak g^*$.
By using an orthonormal basis $e_1,\ldots,e_{\dim \mathfrak g}\in \mathfrak g$ of $\mathfrak g$,
$\bf 1$ can be described as
$${\bf 1}=\sum_{i=1}^{\dim\mathfrak g}e_i\otimes e_i.$$
$\bf 1\in \mathfrak g\otimes \mathfrak g$ is invariant under the diagonal action of $\pi_1(M)$.
Thus we have a natural transformation
$$c:\underline{\RR}\to E_{\rho}\otimes E_{\rho}, 1\mapsto \bf 1.$$
Here $\underline{\RR}$ is the trivial local system, namely a local system
corresponding to the 1-dimensional trivial representation of $\pi_1(M)$.

We define a natural transformation 
$${\rm Tr}:E_{\rho}\otimes E_{\rho}\otimes E_{\rho} \to \underline{\RR}$$
as follows: for $x,y,z\in \mathfrak g$,
$${\rm Tr}(x\otimes y\otimes z)=\langle [x,y],z\rangle$$
where $\langle,\rangle$ is the Killing form and $[,]$ is the Lie bracket.

Let $\pi_1,\pi_2:M^2\to M$ be the projections defined by
$\pi_1(x_1,x_2)=x_1, \pi_2(x_1,x_2)=x_2$.
$\pi_1^*E_{\rho}\otimes \pi_2^*E_{\rho}$ is a local system on $M^2$.
We denote $E_{\rho}\boxtimes E_{\rho}=\pi_1^*E_{\rho}\otimes \pi_2^*E_{\rho}$.
We remark that $E_{\rho}\boxtimes E_{\rho}|_{\Delta}=
E_{\rho}\otimes E_{\rho}$.
The pull-back
$$F_{\rho}=q^*(E_{\rho}\boxtimes E_{\rho})$$
is a local system on $C_2(M)$.
Clearly, $F_{\rho}|_{\partial C_2(M)}=q^*(E_{\rho}\otimes E_{\rho})$.
\subsection{The involution $T$ on $C_2(M)$}
The involution $T_0:M^2\to M^2$ defined by $T_0(x_1,x_2)=(x_2,x_1)$ induces
an involution $T:C_2(M)\to C_2(M)$.
$T_0,T$ induce homomorphisms $T_0^*, T^*$ on the cohomology groups 
$H^*(M^2,E_{\rho}\boxtimes E_{\rho})$, $H^*(C_2(M);F_{\rho})$, $H^*(\Delta;E_{\rho}\otimes E_{\rho})$ and on the space of differential $k$ forms $\Omega^k(C_2(M);F_{\rho})$.
We denote by $H^*_+(M^2;E_{\rho}\boxtimes E_{\rho})$, $H^*_-(M^2;E_{\rho}\boxtimes E_{\rho})$ the 
$+1,-1$ eigenspaces of the homomorphism $T^*_0$ respectively.
We use similar notations $H^*_+(C_2(M);F_{\rho}), H^*_+(\Delta,E_{\rho}\otimes E_{\rho}),
\Omega^k_+(C_2(M);F_{\rho})...$ in the same manner.

Let $T_{S^2}:S^2\to S^2$ be the involution defined as $T_{S^2}(x)=-x$ for any $x\in S^2$.
Since the metric on $M$ is compatible with the framing, we have 
$p\circ T|_{\partial C_2(M)}=T_{S^2}\circ p:\partial C_2(M)\to S^2$.
\subsection{The invariant}
Take a 2-form $\omega_{S^2}\in\Omega^2(S^2;\RR)$ on $S^2$ satisfying $\int_{S^2}\omega_{S^2}=1$
and $T_{S^2}^*\omega_{S^2}=-\omega_{S^2}$.
$p^*\omega_{S^2}$ is a closed 2-form on $\partial C_2(M)$.
Thus $$c_*(p^*\omega_{S^2})=p^*\omega_{S^2}{\bf 1}$$  
is a closed 2-form on $\partial C_2(M)$ such that $(T|_{C_2(M)})^*p^*\omega_{S^2}{\bf 1}=-p^*\omega_{S^2}{\bf 1}$.
Therefore the closed 2-form $p^*\omega_{S^2}{\bf 1}$ represents a cohomology class in 
$H^2_-(\partial C_2(M);F_{\rho}|_{\partial C_2(M)})$:
$$[p^*\omega_{S^2}{\bf 1}]\in H_-^2(\partial C_2(M);F_{\rho}|_{\partial C_2(M)}).$$
\begin{prop}\label{prop21}
There exist 2 forms
$\omega\in \Omega^2(C_2(M);F_{\rho})$ and
$\xi\in \Omega^2(\Delta;E_{\rho}\otimes E_{\rho})$ satisfying the following conditions:
\begin{itemize}
\item[(1)] $\omega|_{\partial C_2(M)}=p^*\omega_{S^2}{\bf 1}+q^*\xi$,
\item[(2)] $T^*\omega=-\omega, (T_0|_{\Delta})^*\xi=-\xi$,
namely $\omega\in \Omega^2_-(C_2(M);F_{\rho}), \xi\in \Omega^2_-(\Delta;E_{\rho}\otimes E_{\rho})$.
\end{itemize}
Furthermore, the cohomology class $[\xi]\in H^2_-(\Delta;E_{\rho}\otimes E_{\rho})$ 
is independent of the choice of $\xi$.
\end{prop}
This proposition is proved in Section~3. 

Now, we have the following 2-forms:
$$q^*\pi_1^*\xi\in \Omega^2(C_2(M);q^*(E_{\rho}^{\otimes 2}\boxtimes
\underline{\RR})),$$
$$q^*\pi_2^*\xi\in \Omega^2(C_2(M);q^*(\underline{\RR}\boxtimes E_{\rho}^{\otimes 2})).$$
Then we obtain closed 6-forms
$$\omega^3\in \Omega^6(C_2(M);F_{\rho}^{\otimes 3}),$$  
$$(q^*\pi_1^*\xi)(q^*\pi_2^*\xi)\omega\in \Omega^6(C_2(M);F_{\rho}^{\otimes 3}).$$  
Since $F_{\rho}^{\otimes 3}=q^*(E_{\rho}^{\otimes 3}\boxtimes E_{\rho}^{\otimes 3})$,
the natural transformation ${\rm Tr}:E_{\rho}^{\otimes 3}\to \underline{\RR}$ induces a natural transformation
$$\tr:F_{\rho}^{\otimes 3}\to (\underline{\RR}\boxtimes\underline{\RR}=)\underline{\RR}.$$
Therefore we get closed 6-forms
$$\tr\omega^3,\tr((q^*\pi_1^*\xi)(q^*\pi_2^*\xi)\omega)\in \Omega^6(C_2(M);\RR).$$
\begin{defini}
$$Z_{\Theta}(\omega)=\int_{C_2(M)}\tr\omega^3,
\Zdam(\omega,\xi)=\int_{C_2(M)}\tr((q^*\pi_1^*\xi)(q^*\pi_2^*\xi)\omega),$$
$$Z_1(M,\rho)=Z_{\Theta}(\omega)-3\Zdam(\omega,\xi).$$
\end{defini}
\begin{theorem}\label{theorem}
$Z_1(M,\rho)$ is an invariant of $M$, $\rho$
(independent of the choices of $\omega,\xi$).
Furthermore, $Z_1(M,\rho)$ is invariant under homotopy of the framing.
\end{theorem}
This theorem is proved in Section~4.
\begin{remark}
When we can take $\xi=0$, obviously $\Zdam(\omega,\xi)=0$ and then $Z_1(M,\rho)$ coincides with the $\Theta$-invariant $I_{(\Theta,{\rm tr},{\rm tr})}(M)$ of the framed 3-manifold $M$ given in Theorem 2.5 in \cite{BC2}.
\end{remark}

\section{Proof of Proposition\ref{prop21}}
In the following commutative diagram, 
the top horizontal line is a part of the long exact sequence of the pair $(C_2(M),\partial C_2(M))$ and the 
bottom line is that of $(M^2,\Delta)$.
Thanks to the excision theorem, the right vertical homomorphism $q^*$ is an isomorphism. 
$$\xymatrix{
 H^2_-(\partial C_2(M);q^*(E_{\rho}\otimes E_{\rho})) \ar[r]^(0.5){\delta^*_{C_2(M)}} \ar@{}[dr]|\circlearrowleft 
 &H^3_-(C_2(M),\partial C_2(M);F_{\rho})\\
 H^2_-(\Delta;E_{\rho}\otimes E_{\rho})\ar[u]^{(q|_{\partial C_2(M)})^*} \ar[r]^(0.45){\delta^*_{M^2}}_(0.45){\cong}& 
  H^3_-(M^2,\Delta;E_{\rho}\otimes E_{\rho})\ar[u]^{q^*}_{\cong}
}$$
Since $H^2_-(M^2;E_{\rho}\otimes E_{\rho})=H^3_-(M^2;E_{\rho}\otimes E_{\rho})=0$,
the connected homomorphism $\delta^*_{M^2}$ on the bottom line is an isomorphism.
Set 
$$ \Phi=(\delta^*_{M^2})^{-1}\circ (q^*)^{-1}\circ\delta^*_{C_2(M)}:
H^2_-(\partial C_2(M);q^*(E_{\rho}\otimes E_{\rho}))\to H^2_-(\Delta;E_{\rho}\otimes E_{\rho}).$$ 
We take a 2-form $\xi\in \Omega_-^2(\Delta;E_{\rho}\otimes E_{\rho})$ such that
$$\Phi([p^*\omega_{S^2}{\bf 1}])=-[\xi]\in H^2_-(\Delta;E_{\rho}\otimes E_{\rho}).$$ 
The above diagram implies that $\Phi(q^*[\xi])=[\xi]$.
Then $\Phi(p^*\omega_{S^2}{\bf 1}+q^*\xi)=0$.
Thus $\delta_{C_2(M)}^*(p^*\omega_{S^2}{\bf 1}+q^*\xi)=0$.
Therefore there exists a closed 2-form $\omega\in \Omega^2_-(C_2(M);F_{\rho})$ such that 
$$\omega|_{\partial C_2(M)}=p^*\omega_{S^2}{\bf 1}+q^*\xi.$$
The second assertion is a direct consequence of the definition $-[\xi]=\Phi([p^*\omega_{S^2}{\bf 1}])$.

\section{Proof of Theorem~\ref{theorem}}
The proof is reduced to the following two propositions:
\begin{prop}\label{step1}
Let $\omega,\omega'\in \Omega_-^2(C_2(M);F_{\rho})$ be closed 2-forms such that
$$\omega|_{\partial C_2(M)}=\omega'|_{\partial C_2(M)}=p^*\omega_{S^2}{\bf 1}+q^*\xi.$$
Then $Z_{\Theta}(\omega)=Z_{\Theta}(\omega')$ and $\Zdam(\omega,\xi)=\Zdam(\omega',\xi)$ hold.
\end{prop}
\begin{prop}\label{step2}
Let $\omega_{S^2,0},\omega_{S^2,1}\in \Omega^2(S^2;\RR)$ be closed 2-forms satsfying $\int_{S^2}\omega_{S^2,0}=\int_{S^2}\omega_{S^2,1}=1$,
$T_{S^2}^*\omega_{S^2,0}=-\omega_{S^2,0}$ and
$T_{S^2}^*\omega_{S^2,1}=-\omega_{S^2,1}$.
Let $\{p_t:\Delta\times S^2\to S^2\}_{t\in[0,1]}$ be a homotopy such that $p_0=p$ and
$p_t\circ T|_{\partial C_2(M)}=T_{S^2}\circ p_t$ for $t=0,1$.
Let $\omega_0,\omega_1\in \Omega_-^2(C_2(M);F_{\rho})$ and 
$\xi_0,\xi_1\in \Omega^2_-(\Delta;E_{\rho}\otimes E_{\rho})$ be closed 2-forms satisfying
$$\omega_0|_{\partial C_2(M)}=p_0^*\omega_{S^2,0}{\bf 1}+q^*\xi_0,
\omega_1|_{\partial C_2(M)}=p_1^*\omega_{S^2,1}{\bf 1}+q^*\xi_1.$$
Then $Z_{\Theta}(\omega_0)-3\Zdam(\omega_0,\xi_0)
=Z_{\Theta}(\omega_1)-3\Zdam(\omega_1,\xi_1)$ holds.
\end{prop}
\subsection{Proof of Proposition\ref{step1}}
\begin{lemma}\label{43}
There exists a 1-form $\eta\in \Omega_-^1(M^2;E_{\rho}\boxtimes E_{\rho})$ such that
$\omega-\omega'=d(q^*\eta)$.
 \end{lemma}
\begin{proof}
In the following diagram, 
the top horizontal line is a part of the long exact sequence of the pair 
$(C_2(M),\partial C_2(M))$ and the 
bottom line is that of $(M^2,\Delta)$.
The left vertical homomorphism $q^*$ is an isomorphism because of the excision theorem. 
$$\xymatrix{
 H^2_-(C_2(M),\partial C_2(M);F_{\rho}) \ar[r] \ar@{}[dr]|\circlearrowleft 
 &H^2_-(C_2(M);F_{\rho})\\
 H^2_-(M^2,\Delta;E_{\rho}\boxtimes E_{\rho})\ar[u]^{q^*}_{\cong} \ar[r]& 
  H^2_-(M^2;E_{\rho}\boxtimes E_{\rho})\ar[u]^{q^*}
}$$
The closed 2-form $\omega-\omega'$ gives a cohomology class in $H^2_-(C_2(M),\partial C_2(M);F_{\rho})$ and
then $((q^*)^{-1}(\omega-\omega'))|_{M^2}$ gives a cohomology class in $H^2_-(M^2;E_{\rho}\boxtimes E_{\rho})$.
Since $H^2_-(M^2;E_{\rho}\boxtimes E_{\rho})=0$,
there exists a 1-form $\eta\in \Omega_-^1(M^2;E_{\rho}\boxtimes E_{\rho})$ such that
$$d\eta=((q^*)^{-1}(\omega-\omega'))|_{M^2}.$$  
Thus we have $d(q^*\eta)=\omega-\omega'$.
\end{proof}
Thanks to Lemma~\ref{43} and Stokes's theorem,
\begin{eqnarray*}
Z_{\Theta}(\omega)-Z_{\Theta}(\omega')
&=& \int_{C_2(M)}\tr\left((\omega-\omega')(\omega^2+\omega\omega'+\omega'^2)\right)\\
&=&
 \int_{C_2(M)}\tr\left(d(q^*\eta)(\omega^2+\omega\omega'+\omega'^2)\right)\\
&=&
 \int_{\partial C_2(M)}\tr\left((q^*\eta)|_{\partial C_2(M)}(\omega^2+\omega\omega'+\omega'^2)|_{\partial C_2(M)}\right)\\
&=& 3\int_{\partial C_2(M)}\tr\left( (q^*\eta)|_{\partial C_2(M)}(p^*\omega_{S^2}{\bf 1}+q^*\xi)^2 \right)\\
&=& 6\int_{\Delta\times S^2}\tr\left( q^*(\eta|_{\Delta})p^*\omega_{S^2}{\bf 1}q^*\xi \right)\\
&=& 6\int_{\Delta}\tr\left(\eta|_{\Delta}\xi{\bf 1}\right).
\end{eqnarray*}
To simplify the notation, we set $\overline{\eta}=\eta|_{\Delta}$.

Let $l:E_{\rho} \otimes E_{\rho}\to E_{\rho}$
be a natural transformation induced from
the Lie bracket $[,]:\mathfrak g\otimes\mathfrak g\to \mathfrak g$.
We have $l(\overline{\eta})\in \Omega^1(\Delta;E_{\rho})$,
$l(\xi)\in \Omega^2(\Delta;E_{\rho})$.
Let $I:E_{\rho}\otimes E_{\rho}\to \underline{\RR}$ be a natural transformation induced from the
inner product of $\mathfrak g$.
Then $I(l(\overline{\eta})l(\xi))$ is a 2-form in $\Omega^2(\Delta;\RR)$.
\begin{lemma}\label{keylemma}
$\tr(\overline{\eta}\xi{\bf 1})=I(l(\overline{\eta})l(\xi))$.
\end{lemma}
\begin{proof}
Since $T_0|_{\Delta}={\rm id}$, 
$\Omega^*_-(\Delta;E\otimes E))=\Omega^*(\Delta;(E\otimes E)_-)$.
Then we only need to check the claim on
 $\mathfrak g^{\otimes 3}\otimes \mathfrak g^{\otimes 3}$.
Let $e_1,\ldots,e_{\dim \mathfrak g}\in \mathfrak g$ be an orthonormal basis of $\mathfrak g$.
 Then $\{e_i\otimes e_j-e_j\otimes e_i\mid i<j\}$ is a basis of $(\mathfrak g\otimes \mathfrak g)^-$.
 It is enough to show the claim for this basis.
 \begin{eqnarray*}
 &&\tr\left((e_i\otimes e_j-e_j\otimes e_i)\otimes(e_k\otimes e_l-e_l\otimes e_k)\otimes(\sum_ne_n\otimes e_n)\right)\\
 &=&2(\langle [e_i,e_k],[e_j,e_l]\rangle-\langle [e_i,e_l],[e_j,e_k]\rangle)\\
 &=&2(\langle e_i,[e_k,[e_j,e_l]]\rangle+\langle e_i,[e_l,[e_k,e_j]]\rangle)\\
 &=&2\left((-\langle e_i,[e_j,[e_l,e_k]]\rangle-\langle e_i,[e_l,[e_k,e_j]]\rangle)+\langle e_i,[e_l,[e_k,e_j]]\rangle\right)\\ &=&2\langle e_i,[e_j,[e_k,e_l]]\rangle\\
 &=&2\langle [e_i,e_j],[e_k,e_l]\rangle\\
 &=&\frac{1}{2}\langle 2[e_i,e_j],2[e_k,e_l]\rangle\\
 &=&\frac{1}{2}\langle l(e_i\otimes e_j-e_j\otimes e_i)l(e_k\otimes e_l-e_l\otimes e_k)\rangle. 
 \end{eqnarray*}
\end{proof}
\begin{cor}
$\int_{\Delta}\tr(\overline{\eta}\xi{\bf 1})=0$.
\end{cor}
\begin{proof}
Thanks to the above lemma,
$$\int_{\Delta}\tr(\overline{\eta}\xi{\bf 1})=\int_{\Delta}I(l(\overline{\eta})l(\xi)).$$
Since $E_{\rho}$ is acyclic, $[l(\xi)]=0\in H^2(\Delta;E_{\rho})=0$.
Thus there exists a 1-form $\zeta\in\Omega^1(\Delta;E_{\rho})$ such that
$d\zeta=l(\xi)$.
Therefore
\begin{eqnarray*}
\int_{\Delta}I(l(\overline{\eta})l(\xi))&=&\int_{\Delta}I(\overline{\eta}d\zeta)\\
&=&\int_{\Delta}dI(\overline{\eta}\zeta)=0.
\end{eqnarray*}
\end{proof}
Thanks to the above lemma, we have
$$Z_{\Theta}(\omega)-Z_{\Theta}(\omega')=0.$$
Similarly,
\begin{eqnarray*}
\Zdam(\omega,\xi)-\Zdam(\omega',\xi)&=&
\int_{C_2(M)}\tr((q^*\pi_1^*\xi)(q^*\pi_2^*\xi)(\omega-\omega'))\\
&=&
\int_{C_2(M)}\tr((q^*\pi_1^*\xi)(q^*\pi_2^*\xi) dq^*\eta)\\
&=&\int_{\partial C_2(M)}\tr\left(q^*((\pi_1|_{\Delta})^*\xi(\pi_2|_{\Delta})^*\xi\overline{\eta})\right).
\end{eqnarray*}
Since  $(\pi_1|_{\Delta})^*\xi(\pi_2|_{\Delta})^*\xi\overline{\eta}$ is a 5-form on
the 3-dimensional manifold $\Delta$,
the last term is vanishing.
This completes the proof of Proposition~\ref{step1}.
\subsection{Proof of Proposition~\ref{step2}}
Since $[\omega_{S^2,0}]=[\omega_{S^2,1}]\in H^2(S^2;\RR)$,
there exists a closed 2-form $\widetilde{\omega_{S_2}}\in \Omega^2([0,1]\times S^2;\RR)$ such that $\widetilde{\omega_{S^2}}|_{\{t\}\times S^2}=\omega_{S^2,t}$  for $t=0,1$.

Since $[\xi_0]=[\xi_1]$(Proposition~\ref{prop21}), there exists a closed 1-form 
$$\widetilde{\xi}\in \Omega^1([0,1]\times\Delta,\pi_{\Delta}^*(E_{\rho}\otimes E_{\rho}))$$
such that $\widetilde{\xi}|_{\{0\}\times\Delta}=\xi_0$ and $\widetilde{\xi}|_{\{1\}\times\Delta}=\xi_1$.
Here $\pi_{\Delta}:[0,1]\times\Delta\to \Delta$ is the projection.

Let $\pi_{C_2(M)}:[0,1]\times C_2(M)\to C_2(M)$ be the projection.
Let $\widetilde{q}={\rm id}_{[0,1]}\times q:[0,1]\times C_2(M)\to [0,1]\times M^2$
and we also denote the restriction map $\widetilde{q}|_{[0,1]\times \partial C_2(M)}:[0,1]\times \partial C_2(M)\to [0,1]\times \Delta$ as $\widetilde q$.
By a similar argument as in Proposition~\ref{prop21}, we can take a closed 2-form 
$$\widetilde{\omega}\in \Omega^2([0,1]\times C_2(M),\pi_{C_2(M)}^*F_{\rho})$$
such that 
$$\widetilde{\omega}|_{[0,1]\times \partial C_2(M)}=\widetilde{p}^*\widetilde{\omega_{S^2}}{\bf 1}+\widetilde{q}^*\widetilde{\xi}.$$
Here $\widetilde{p}=\{p_t\}_t:([0,1]\times \partial C_2(M)=)[0,1]\times\Delta\times S^2\to S^2$ is the homotopy between $p_0$ and $p_1$.

Thanks to Proposition~\ref{step1}, both $Z_{\Theta}(\omega)$ and $\Zdam(\omega,\xi)$
depend only on $\omega|_{\Delta\times S^2}$ and $\xi$.
Thus we have
$Z_{\Theta}(\omega_0)=Z_{\Theta}(\widetilde{\omega}|_{\{0\}\times C_2(M)})$,
$Z_{\Theta}(\omega_1)=Z_{\Theta}(\widetilde{\omega}|_{\{1\}\times C_2(M)})$,
$\Zdam(\omega_0,\xi_0)=\Zdam(\widetilde{\omega}|_{\{0\}\times C_2(M)},\xi_0)$ and
$\Zdam(\omega_1,\xi_1)=\Zdam(\widetilde{\omega}|_{\{1\}\times C_2(M)},\xi_1)$.
We note that, with our orientation convention, 
$$\partial ([0,1]\times C_2(M))=\{1\}\times C_2(M)-\{0\}\times C_2(M)-[0,1]\times \partial C_2(M).$$
Therefore, by using Stokes' theorem, 
\begin{eqnarray*}
0&=&\int_{[0,1]\times C_2(M)}d\tr\widetilde{\omega}^3\\
&=&\int_{\{1\}\times C_2(M)}\tr(\widetilde{\omega}|_{\{1\}\times C_2(M)}^3)
-\int_{\{0\}\times C_2(M)}\tr(\widetilde{\omega}|_{\{0\}\times C_2(M)}^3)\\
&&-\int_{[0,1]\times\partial C_2(M)}\tr(\widetilde{\omega}|_{[0,1]\times\partial C_2(M)}^3)\\
&=&Z_{\Theta}(\widetilde{\omega}|_{\{1\}\times C_2(M)})-Z_{\Theta}(\widetilde{\omega}|_{\{0\}\times C_2(M)})-\int_{[0,1]\times\partial C_2(M)}\tr(\widetilde{p}^*\widetilde{\omega_{S^2}}{\bf 1}+\widetilde{q}^*\widetilde{\xi})^3\\
&=&
Z_{\Theta}(\omega_1)-Z_{\Theta}(\omega_0)
-\int_{[0,1]\times\partial C_2(M)}\tr(3\widetilde{p}^*\widetilde{\omega_{S^2}}{\bf 1}\widetilde{q}^*\widetilde{\xi}^2)
\end{eqnarray*}
We denote $\widetilde{\pi_i}=
{\rm id}_{[0,1]}\times\pi_i:[0,1]\times M^2\to [0,1]\times M$ for $i=1,2$.
We have,
\begin{eqnarray*}
0&=&\int_{[0,1]\times C_2(M)}d\tr\left(
(\widetilde{q}^*\widetilde{\pi_1}^*\widetilde{\xi})
(\widetilde{q}^*\widetilde{\pi_2}^*\widetilde{\xi})
\widetilde{\omega}\right)\\
&=&\Zdam(\omega_1,\xi_1)-\Zdam(\omega_0,\xi_0)\\
&&-
\int_{[0,1]\times \partial C_2(M)}\tr\left(
(\widetilde{q}^*((\widetilde{\pi_1}|_{[0,1]\times\Delta})^*\xi
(\widetilde{\pi_2}|_{[0,1]\times \Delta})^*\xi)\widetilde{\omega}|_{[0,1]\times\partial C_2(M)}\right).
\end{eqnarray*}
Here, $\widetilde{\pi_1}|_{[0,1]\times\Delta}=\widetilde{\pi_2}|_{[0,1]\times \Delta}
:[0,1]\times \Delta\to M$.
Thus $(\widetilde{\pi_1}|_{[0,1]\times \Delta})^*\widetilde{\xi}(\widetilde{\pi_2}|_{[0,1]\times \Delta})^*\widetilde{\xi}=\widetilde{\xi}^2$ under the identification $\Delta=M$.
We have
\begin{eqnarray*}
&&\Zdam(\omega_1,\xi_1)-\Zdam(\omega_0,\xi_0)\\
&=&\int_{[0,1]\times \partial C_2(M)}\tr(\widetilde{p}^*\widetilde{\omega_{S^2}}{\bf 1}\widetilde{q}^*\widetilde{\xi}^2)
\end{eqnarray*}
Then we have
$$Z_{\Theta}(\omega_1)-Z_{\Theta}(\omega_0)=3(\Zdam(\omega_1,\xi_1)-\Zdam(\omega_0,\xi_0)).$$
This completes the proof of Proposition~\ref{step2}.

\section{A framing correction}
In this section, we introduce a correction term for framings to give an invariant of closed 3-manifolds with acyclic representations.
Let $M$ be a closed oriented 3-manifold (without framings).
Recall that $\partial C_2(M)$ is identified with the unit sphere bundle $STM$
(see Section \ref{sec21}).
Take a framing $f:TM\to M\times \RR^3$ of $M$.
Then $(M,f)$ is a framed 3-manifold.
Let $p:(\partial C_2(M)=)STM\to S^2$ be the projection defined by the framing $f$.
Let $\delta(f)\in \ZZ$ be the signature defect (or Hirzebruch defect. For example,
see \cite{Atiyah}, \cite{KM} for the details) of a framing $f$.
For the convenience of the reader, we give a short review of the construction of $\delta(f)$ in the next section.
Let $\rho:\pi_1(M)\to G$ be an acyclic representation as in Section~\ref{sec21}.
\begin{theorem}\label{propfr}
$$Z_1((M,f),\rho)-(\dim \mathfrak g)^2\delta(f)$$
is an topological invariant of $M,\rho$.
\end{theorem}
\subsection{The signature defect $\delta(p)$}
Let $W$ be a compact 4-manifold such that $\partial W=M$ and
its Euler characteristic is zero.
Then there exists an $\RR^3$ sub-bundle $T^vW$ of $TW$ satisfying $T^vW|_M=TM$.
Let $ST^vW\to W$ be the unit sphere bundle of $T^vW\to W$.
Thus $ST^vW$ is a 6-dimensional manifold with $\partial ST^vW=STM$. 
We denote by $F_W\to ST^vW$ the tangent bundle along the fiber of the $S^2$ bundle $\pi:ST^vW\to W$.

Take a closed 2-form $\alpha_W\in \Omega^2(ST^vW;\RR)$ such that
$ \alpha_W|_{STM}=p^*\omega_{S^2}$ and $[\alpha_W]=e(F_W)/2\in H^2(ST^vW;\RR)$,
where $e(F_W)$ is the Euler class of $F_W\to ST^vW$.
\begin{lemma}
When $\partial W=M=\emptyset$, we have $\int_{ST^vW}\alpha_W^3=\frac{3}{4}{\rm Sign}W$.
Here $\Sign W$ is the signature of $W$.  
\end{lemma}
\begin{proof}
We give an outline of the proof.
See Appendix of \cite{shimizu} or Proposition~2.45 of \cite{Lescop}, for the details of the proof.

Since $W$ is closed, $\int_{ST^vW}\alpha_W^3=\int_{ST^vW}\left(\frac{1}{2}e(F_W)\right)^3$.
We denote by $p_1(F_W)\in H^4(ST^vW;\RR)$ the 1st Pontrjagin class of the bundle $F_W$. We remark that $\underline{\RR}\oplus F_W=\pi^*T^vW$ and 
$\underline{\RR}\oplus  T^vW=TW$. Here $\underline{\RR}$ is the trivial $\RR$ bundle over an appropriate manifold.
Therefore,
\begin{eqnarray*}
\int_{ST^vW}\alpha_W^3&=&\frac{1}{8}\int_{ST^vW}e(F_W)^3\\
&=&\frac{1}{8}\int_{ST^vW}e(F_W)p_1(F_W)\\
&=&\frac{1}{8}\int_{ST^vW}e(F_W)\pi^*p_1(T^vW)\\
&=&\frac{1}{4}\int_{W}p_1(TW)\\
&=&\frac{3}{4}\Sign W.
\end{eqnarray*}
\end{proof}
Thanks to the Novikov additivity for the signature, the following corollary holds. 
\begin{cor}
The signature defect $\delta(f)$ defined to be 
$\delta(f)=\int_{ST^W}\alpha_W^3-\frac{3}{4}\Sign W$
is independent of the choices of $W$ and $\alpha_W$.  
\end{cor}
\subsection{Proof of Theorem~\ref{propfr}}
Let $f_0,f_1:TM\to M\times \RR^3$ be framings and
let $p_0,p_1:\partial C_2(M)\to S^2$ be the projections given by framings $f_0,f_1$ respectively.
Since $[p_0^*\omega_{S^2}]$ and $[p_1^*\omega_{S^2}]$ are in $H^2_-(\Delta\times S^2;\RR)=H^2(S^2;\RR)=\RR$,
$[p_0^*\omega_{S^2}]=[p_1^*\omega_{S^2}]$.
Thus there exists a closed 2-form 
$$\widetilde{\omega}_{\partial}\in\Omega^2_-([0,1]\times\partial C_2(M);\RR)$$
such that 
$\widetilde{\omega}_{\partial}|_{\{0\}\times\partial C_2(M)}=p_0^*\omega_{S^2},$
$\widetilde{\omega}_{\partial}|_{\{1\}\times\partial C_2(M)}=p_1^*\omega_{S^2}.$
We recall that $\xi\in\Omega^2_-(\Delta;E_{\rho}\otimes E_{\rho})$ is a closed 2-form representing $\Phi([p^*\omega_{S^2}{\bf 1}])=\Phi \circ c^*([p^*\omega_{S^2}])$ when we take a projection $p:\partial C_2(M)\to S^2$ given by a framing $f$.
The homomorphism $\Phi\circ c^*$ is independent from the choice of a framing.
Then we can use same $\xi\in\Omega^2_-(\Delta;E_{\rho}\otimes E_{\rho})$ for any framing.

By a similar argument as in proof of Proposition\ref{prop21}, we can take a closed 2-form 
$$\widetilde{\omega}\in \Omega^2([0,1]\times C_2(M);\pi_{C_2(M)}^*F_{\rho})$$
 such that 
$$\widetilde{\omega}|_{[0,1]\times \partial C_2(M)}=\widetilde{\omega}_{\partial}{\bf 1}+Q^*\xi.$$
Here, $\pi_{C_2(M)}:[0,1]\times C_2(M)\to C_2(M)$ and
$Q:[0,1]\times\partial C_2(M)\to \Delta$ are the projections.
We denote by 
$$\omega_0=\widetilde{\omega}|_{\{0\}\times C_2(M)},$$
$$\omega_1=\widetilde{\omega}|_{\{1\}\times C_2(M)}.$$
Then, 
$$Z_1((M,f_0),\rho)=Z_{\Theta}(\omega_0)-3\Zdam(\omega_0,\xi),$$
$$Z_1((M,f_1),\rho)=Z_{\Theta}(\omega_1)-3\Zdam(\omega_1,\xi).$$
Thanks to Stokes' theorem, 
\begin{eqnarray*}
0&=&\int_{[0,1]\times C_2(M)}d\tr(\widetilde{\omega}^3)\\
&=&Z_{\Theta}(\omega_1)-Z_{\Theta}(\omega_0)
-\int_{[0,1]\times \partial C_2(M)}\tr(\widetilde{\omega}|_{[0,1]\times\partial C_2(M)}^3)\\
&=&Z_{\Theta}(\omega_1)-Z_{\Theta}(\omega_0)
-\int_{[0,1]\times \partial C_2(M)}\tr(\widetilde{\omega}_{\partial}^3{\bf 1}^{\otimes 3})
-\int_{[0,1]\times \partial C_2(M)}3\tr(\widetilde{\omega}_{\partial}^2{\bf 1}^{\otimes 2}Q^*\xi)\\
&=&
Z_{\Theta}(\omega_1)-Z_{\Theta}(\omega_0)
-\int_{[0,1]\times \partial C_2(M)}\widetilde{\omega}_{\partial}^3\tr({\bf 1}^{\otimes 3})
-\int_{[0,1]\times \partial C_2(M)}3\tr(\widetilde{\omega}_{\partial}^2{\bf 1}^{\otimes 2}Q^*\xi)\\
&=&
Z_{\Theta}(\omega_1)-Z_{\Theta}(\omega_0)
-(\dim \mathfrak g)^2\int_{[0,1]\times \partial C_2(M)}\widetilde{\omega}_{\partial}^3
-\int_{[0,1]\times \partial C_2(M)}3\tr(\widetilde{\omega}_{\partial}^2{\bf 1}^{\otimes 2}Q^*\xi).
\end{eqnarray*}
We denote $\overline{\pi}_i:[0,1]\times M^2\to M,
(t,x_1,x_2)\mapsto x_i$ for $i=1,2$.
We have,
\begin{eqnarray*}
0&=&\int_{[0,1]\times C_2(M)}d\tr\left(
(\widetilde{q}^*\overline{\pi}_1^*\xi)
(\widetilde{q}^*\overline{\pi}_2^*\xi)
\widetilde{\omega}\right)\\
&=&\Zdam(\omega_1,\xi)-\Zdam(\omega_0,\xi)\\
&&-
\int_{[0,1]\times \partial C_2(M)}\tr\left(
(\widetilde{q}^*((\overline{\pi}_1|_{[0,1]\times\Delta})^*\xi
(\overline{\pi}_2|_{[0,1]\times \Delta})^*\xi)\widetilde{\omega}_{\partial}{\bf 1}\right)\\
&=&
\Zdam(\omega_1,\xi)-\Zdam(\omega_0,\xi)-
\int_{[0,1]\times \partial C_2(M)}\tr\left(
Q^*\xi^2\widetilde{\omega}_{\partial}{\bf 1}\right)\\
&=&0
\end{eqnarray*}
Thus we have
$$Z_1((M,f_0),\rho)-Z_1((M,f_1),\rho)
=(\dim\mathfrak g)^2\int_{[0,1]\times \partial C_2(M)}\widetilde{\omega}_{\partial}^3
+\int_{[0,1]\times \partial C_2(M)}3\tr(\widetilde{\omega}_{\partial}^2{\bf 1}^{\otimes 2}Q^*\xi).$$
\begin{lemma}
$$\tr(\widetilde{\omega}_{\partial}{\bf 1}^{\otimes 2}Q^*\xi)=0.$$
\end{lemma}
\begin{proof}
Let $T_E:E_{\rho}\otimes E_{\rho}\to E_{\rho}\otimes E_{\rho}$ be the involution 
induced by $\mathfrak g\otimes \mathfrak g\to \mathfrak g\otimes \mathfrak g,
x\otimes y\mapsto y\otimes x$.
Cleary, $\tr\circ T_E^{\otimes 3}=\tr:E^{\otimes 3}\otimes E^{\otimes 3}\to \underline{\RR}$.
Since $T_E({\bf 1})={\bf 1}$ and $T_E^*=(T_0|_{\Delta})^*$ on $\Omega^1(\Delta;E_{\rho}\otimes E_{\rho})$, we have
$$\tr(\widetilde{\omega}_{\partial}{\bf 1}^{\otimes 2}Q^*\xi)=\tr\left(T_E^{\otimes 3}(\widetilde{\omega}_{\partial}{\bf 1}^{\otimes 2}Q^*\xi)\right)
=-\tr(\widetilde{\omega}_{\partial}{\bf 1}^{\otimes 2}Q^*\xi).$$
Thus $\tr(\widetilde{\omega}_{\partial}{\bf 1}^{\otimes 2}Q^*\xi)=0$.
\end{proof}
\begin{lemma} 
$$\delta(f_1)-\delta(f_0)=\int_{[0,1]\times\partial C_2(M)}\widetilde{\omega}_{\partial}^3.$$
\end{lemma}
\begin{proof}
We take a compact 4-manifold $W$ with $\partial W=M$ and its Euler characteristic is zero.
Take a collar neighborhood 
$[0,1]\times\partial M$
of $M=\partial W$ in $W$ such that $\{1\}\times M=\partial W$.
Set $W_0=W\setminus([0,1]\times M)$.
We can take $T^vW$ as $T^vW|_{[0,1]\times M}=[0,1]\times TM$.
Thus $ST^vW|_{[0,1]\times M}$ is identified with $[0,1]\times \partial C_2(M)$.
Take a closed 2-form $\alpha_W\in\Omega^2(ST^vW;\RR)$ satisfying 
$\alpha_W|_{[0,1]\times STM}=\widetilde{\omega}_{\partial}$ and
$[\alpha_W]=\frac{1}{2}e(F_W)$.
Then we have
\begin{eqnarray*}
\delta(f_1)-\delta(f_0)&=&
\left(\int_{ST^vW}\alpha_W^3-\frac{3}{4}\Sign W\right)
-\left(\int_{ST^vW_0}(\alpha_W|_{ST^vW_0})^3-\frac{3}{4}\Sign W_0\right)\\
&=&\int_{[0,1]\times STM}(\alpha_W|_{[0,1]\times STM})^3\\
&=&\int_{[0,1]\times \partial C_2(M)}\widetilde{\omega}_{\partial}^3.
\end{eqnarray*}

\end{proof}
By the above two lemmas,
$$Z_1((M,f_0),\rho)-(\dim\mathfrak g)^2\delta(f_0)
=Z_1((M,f_1),\rho)-(\dim\mathfrak g)^2\delta(f_1).$$
Namely, 
$Z_1((M,f),\rho)-(\dim\mathfrak g)^2\delta(f)$
is independent of the choice of a framing $f$.

\end{document}